\documentclass{amsart}
\usepackage{amssymb}
\usepackage{color}
\usepackage{graphicx}
\theoremstyle{plain}
\newtheorem{theorem}{Theorem}[section]
\newtheorem{corollary}[theorem]{Corollary}

\newtheorem{proposition}[theorem]{Proposition}
\newtheorem{conjecture}[theorem]{Conjecture}
\newtheorem{lemma}[theorem]{Lemma}
\newtheorem{rem}[theorem]{Remark}

\numberwithin{equation}{section}

\newcommand{\Q}{\mathbb{Q}}

\makeatletter
\@namedef{subjclassname@2020}{%
  \textup{2020} Mathematics Subject Classification}
\makeatother

\begin{document}

\title[Divisibility and Arithmetic Properties]
{Divisibility and Arithmetic Properties of a Class of Sparse Polynomials}

\author{Karl Dilcher}
\address{Department of Mathematics and Statistics\\
         Dalhousie University\\
         Halifax, Nova Scotia, B3H 4R2, Canada}
\email{dilcher@mathstat.dal.ca}

\author{Maciej Ulas}
\address{Jagiellonian University, Faculty of Mathematics and Computer Science,
Institute of Mathematics, \L{}ojasiewicza 6, 30-348 Krak\'ow, Poland}
\email{Maciej.Ulas@im.uj.edu.pl}

\thanks{Research of the first author supported in part by the Natural Sciences and Engineering
        Research Council of Canada, Grant \# 145628481. Research of the second author supported by a grant of the Polish National Science Centre, no. UMO-2019/34/E/ST1/00094}

\date{}

\setcounter{equation}{0}

\begin{abstract}
We investigate algebraic and arithmetic properties of a class of sequences of
sparse polynomials that have binomial coefficients both as exponents and as
coefficients. In addition to divisibility and irreducibility results we also
consider rational roots. This leads to the study of an infinite class of
integer sequences which have interesting properties and satisfy linear
recurrence relations.
\end{abstract}

\maketitle

\section{Introduction}\label{sec1}

The sequence of sparse polynomials defined by
\begin{equation}\label{1.1}
f_n(z) := \sum_{j=0}^{n}\binom{n}{j}z^{j(j-1)/2}
\end{equation}
arises naturally from a graph theoretic question related to the expected
number of independent sets of a graph \cite{BDM2}. Various properties,
including asymptotics, zero distribution, and arithmetic properties, can be
found in \cite{BDM1}, \cite{BDM2}, \cite{BDM3}, and \cite{GN}.
More recently, in \cite{DU}, we extended the polynomials in \eqref{1.1} by
introducing the class of polynomials
\begin{equation}\label{1.2}
f_{m,n}(z) := \sum_{j=0}^{n}\binom{n}{j}z^{\binom{j}{m}},
\end{equation}
where we typically fix the integer parameter $m\geq 1$ and consider the
sequence $(f_{m,n}(z))_n$; obviously $f_{2,n}(z)=f_n(z)$. Since
$f_{1,n}(z)=(1+z)^n$, we usually assume that $m\geq 2$. It is also clear from
\eqref{1.2} that $f_{m,n}(z)=2^n$ when $n\leq m-1$, and that for all $m\geq 1$
we have
\begin{equation}\label{1.3}
f_{m,m}(z) = z+2^m-1,\quad f_{m,m+1}(z)=z^{m+1}+(m+1)z+\big(2^{m+1}-m-2\big),
\end{equation}
and we have the special values
\begin{equation}\label{1.4}
f_{m,n}(0) = \sum_{j=0}^{m-1}\binom{n}{j},\qquad f_{m,n}(1)=2^n.
\end{equation}

In \cite{DU} we investigated various analytic properties of the polynomials
$f_{m,n}(z)$, especially monotonicity and log-concavity, connections between
the polynomials and their derivatives, and the distribution of their real and
complex zeros. Some of the properties were obtained for a more general class
of polynomials.

It is the purpose of the present paper to study arithmetic and algebraic
properties of the polynomials $f_{m,n}(z)$, especially divisibility and
irreducibility, and number theoretic properties of special values of
$f_{m,n}(z)$. We begin, in Section~2, by considering the sequence of special
values $(f_{m,n}(-1))_n$; the results in that section will be useful
also in later sections. In Section~3 we investigate divisibility properties
of the polynomials, and Section~4 is devoted to the related concept of rational
roots. In Section~5 we deal with further properties of the sequence
$(f_{m,n}(-1))_n$ in the special case $m=2^k$. Finally, in Section~6,
we prove some irreducibility results.

\section{Monotonicity Properties}

We define the usual difference operator $\Delta$ on a sequence $(a_n)$ by
$\Delta a_n=a_{n+1}-a_n$, and the operator $\Delta^r$ of order $r\geq 0$ is
defined recursively by $\Delta^{r+1}=\Delta\circ\Delta^r$, with
$\Delta^0 a_n=a_n$. A sequence of real numbers is called {\it absolutely
monotonic\/} if for all integers $r,n\geq 0$ we have $\Delta^r a_n \geq 0$.
It is well-known that
\begin{equation}\label{3.2}
\Delta^r a_n = \sum_{k=0}^r(-1)^k\binom{r}{k}a_{n+r-k},
\end{equation}
which is easy to see by induction. This also means that if $a_n=f(n)$, where
$f$ is a polynomial of degree $d$, then for $r>d$ we have $\Delta^r a_n=0$
for all $n\geq 0$.

In \cite{DU} we obtained the following as a consequence of a more general
result; see also Lemma~\ref{lem:4a.3} below.

\begin{proposition}\label{prop:3.1}
For any integer $m\geq 1$ and real $z>0$, the sequence $(f_{m,n}(z))_{n\geq 0}$
is absolutely monotonic.
\end{proposition}

This gives rise to the question whether there are real numbers
$z<0$ and integers $m\geq 2$ such that $(f_{m,n}(z))_{n\geq 0}$ is also an
absolutely monotonic sequence. Computations suggest that in general this is not
the case. However, we have the following surprising result.

\begin{proposition}\label{prop:3.3}
Let $m$ be a positive integer.
\begin{enumerate}
\item[(1)] If $m$ is odd, then the sequence $(f_{m,n}(-1))_{n\geq 1}$ is
absolutely monotonic.
\item[(2)] If $m$ is even, then $(f_{m,n}(-1))_{n\geq 1}$ is not absolutely
monotonic.
\end{enumerate}
\end{proposition}

See Table~1 for an illustration of this result.
In spite of the negative nature of part (2), much more can be
said about the sequence $(f_{m,n}(-1))_{n\geq 1}$ for both even and odd $m$;
this will be done in Section~4.

For the proof of Proposition~\ref{prop:3.3} and for other results in this paper
we require some parity properties of binomial coefficients. We
first quote a special case of a well-known congruence of Lucas.

\begin{lemma}\label{lem:3.4}
Suppose that the integers $0\leq m\leq k$ are given in binary expansion as
$k=a_h\cdot2^h+\cdots+a_1\cdot 2+a_0$ and $m=b_h\cdot2^h+\cdots+b_1\cdot 2+b_0$.
Then
\begin{equation}\label{3.5}
\binom{k}{m}\equiv\binom{a_h}{b_h}\cdots\binom{a_1}{b_1}\binom{a_0}{b_0}
\pmod{2}.
\end{equation}
\end{lemma}

\bigskip
\begin{center}
\begin{tabular}{|r||r|r|r|r|r|}
\hline
$n$ & $f_{2,n}(-1)$ & $f_{3,n}(-1)$ & $f_{4,n}(-1)$ & $f_{5,n}(-1)$ & $f_{6,n}(-1)$ \\
\hline
1 & 2 & 2 & 2 & 2 & 2 \\
2 & 2 & 4 & 4 & 4 & 4 \\
3 & 0 & 6 & 8 & 8 & 8 \\
4 & $-4$ & 8 & 14 & 16 & 16 \\
5 & $-8$ & 12 & 20 & 30 & 32 \\
6 & $-8$ & 24 & 20 & 52 & 62 \\
7 & 0 & 56 & 0 & 84 & 112 \\
8 & 16 & 128 & $-68$ & 128 & 184 \\
9 & 32 & 272 & $-232$ & 188 & 272 \\
10 & 32 & 544 & $-560$ & 280 & 364 \\
11 & 0 & 1056 & $-1120$ & 464 & 464 \\
12 & $-64$ & 2048 & $-1912$ & 928 & 664 \\
\hline
\end{tabular}

\medskip
{\bf Table~1}: $f_{m,n}(-1)$ for $2\leq m\leq 6$ and  $1\leq n\leq 12$.
\end{center}
\bigskip

For the general case, valid for any prime base and modulus $p$ in place of 2,
see, e.g., \cite{Fi} where a proof is also given. The next lemma is related to
``the geometry of binomial coefficients"; see, e.g., \cite{Sv} or \cite{Wo} for
some fractal-like images of Pascal's triangle modulo 2, along with other
related properties. We cannot claim that the following properties are new, but
we provide proofs for the sake of completeness.

\begin{lemma}\label{lem:3.5}
Let the positive integers $m$ and $\nu$ be such that $2^{\nu-1}\leq m<2^\nu$.
Then
\begin{enumerate}
\item[(1)] the sequence $\big(\binom{m+k}{m}\pmod{2}\big)_{k\geq 0}$ is
periodic with period $2^\nu$, but not with period $2^{\mu}$, $\mu<\nu$;
\item[(2)] when $m$ is odd, then $\binom{k}{m}$ and $\binom{k+1}{m}$ cannot
both be odd, for any $k\geq 0$;
\item[(3)] when $m$ is even, there is always an integer $k$,
$2^{\nu}\leq k<2^{\nu+1}$, such that $\binom{k}{m}$ and $\binom{k+1}{m}$ are
both odd.
\end{enumerate}
\end{lemma}

\begin{proof}
(1) Suppose that $m$ has the binary representation as in Lemma~\ref{lem:3.4},
with $b_h=1$. Then $h=\nu-1$, and the residue modulo 2 in \eqref{3.5} does not
change if we add multiples of $2^\nu$ to $k$ since we may take
$b_{h+1}=b_{h+2}=\cdots=0$.

To prove the second statement, we note that
\[
\textstyle{\binom{m}{m}=1,\quad\hbox{while}\quad
\binom{m+2^{\nu-1}}{m}\equiv 0\pmod{2},}
\]
so that we cannot have periodicity modulo $2^{\nu-1}$ or any smaller power of 2.
This last congruence comes from the fact that
$m=2^{\nu-1}+b_{\nu-2}2^{\nu-2}+\cdots$, which implies
$m+2^{\nu-1}=2^{\nu}+b_{\nu-2}2^{\nu-2}+\cdots$; hence the binomial coefficient
in \eqref{3.5} corresponding to $2^{\nu-1}$ is $\binom{0}{1}=0$.

(2) If $m$ is odd, then $b_0=1$ in \eqref{3.5}. Since one of $k, k+1$ is even,
the corresponding $a_0$ is 0, which means that the right-hand side of
\eqref{3.5} is zero, that is, at least one of $\binom{k}{m}$, $\binom{k+1}{m}$
is even.

(3) We take $k=m+2^{\nu}$. Then by part (1), $\binom{k}{m}=\binom{m}{m}=1$ and
$\binom{k+1}{m}=\binom{m+1}{m}=m+1$, both of which are odd since $m$ is even.
\end{proof}

The next lemma is also needed for the proof of Proposition~\ref{prop:3.3}, as
well as for Proposition~\ref{prop:4a.2} later in this paper. It is
actually a special case of Proposition~3.1 in \cite{DU}, but for the sake of
completeness we repeat the proof here. We also note that by \eqref{3.2}, this
lemma immediately implies Proposition~\ref{prop:3.1}.

\begin{lemma}\label{lem:4a.3}
For all integers $m\geq 2$ and $r,n\geq 0$ we have
\begin{equation}\label{4a.4}
\sum_{\nu=0}^{r}(-1)^{\nu}\binom{r}{\nu}f_{m,n+r-\nu}(z)
=\sum_{j=0}^{n}\binom{n}{j}z^{\binom{j+r}{m}}.
\end{equation}
\end{lemma}

\begin{proof}
Using the definition \eqref{1.2}, we rewrite the left-hand side of \eqref{4a.4}
as
\begin{equation}\label{4a.5}
\sum_{\nu=0}^{r}(-1)^{\nu}\binom{r}{\nu}
\sum_{j=0}^{n+r-\nu}\binom{n+r-\nu}{j}z^{\binom{j}{m}}
=\sum_{j=0}^{n+r}\left(
\sum_{\nu=0}^{r}(-1)^{\nu}\binom{r}{\nu}\binom{n+r-\nu}{j}\right)
z^{\binom{j}{m}},
\end{equation}
where we have extended the range of $j$ by adding zero-terms. Now we observe
that, by \eqref{3.2}, the inner sum on the right of \eqref{4a.5} is just
$\Delta^r\binom{n}{j}$, and $\binom{n}{j}$ is a polynomial in $n$ of degree $j$.
Hence, by the remark following \eqref{3.2}, this sum is 0 for $j<r$. When
$j\geq r$, this inner sum has the known evaluation $\binom{n}{j-r}$; see, e.g.,
\cite[Eq.~(3.49)]{Go}. So, altogether the left-hand side of \eqref{4a.4}, with
\eqref{4a.5}, becomes
\[
\sum_{j=r}^{n+r}\binom{n}{j-r}z^{\binom{j}{m}}
= \sum_{j=0}^{n}\binom{n}{j}z^{\binom{r+j}{m}},
\]
which was to be shown.
\end{proof}

\begin{proof}[Proof of Proposition~\ref{prop:3.3}]
We have seen at the beginning of this section that a sequence $(a_n)$ is
absolutely monotonic if and only if the right-hand side of \eqref{3.2} is
non-negative for all $r,n\geq 0$.
In view of \eqref{4a.4}, we denote
\begin{equation}\label{3.6}
S_m(n,r) := \sum_{k=0}^{n}\binom{n}{k}(-1)^{\binom{k+r}{m}}.
\end{equation}
We need to show that for all $n\geq 1$ and $r\geq 0$, we have $S_m(n,r)\geq0$ if
and only if $m$ is odd. For this purpose we show that these sums satisfy a
``triangular" recurrence relation.
Indeed, by manipulating the right-hand side of \eqref{3.6}, we get
\begin{align*}
S_m(n,r) + S_m(n,r+1) &= \sum_{k=0}^{n}\binom{n}{k}(-1)^{\binom{k+r}{m}}
+ \sum_{k=1}^{n+1}\binom{n}{k-1}(-1)^{\binom{k+r}{m}} \\
&= \sum_{k=0}^{n+1}\left(\binom{n}{k}+\binom{n}{k-1}\right)(-1)^{\binom{k+r}{m}}\\
&= \sum_{k=0}^{n+1}\binom{n+1}{k}(-1)^{\binom{k+r}{m}},
\end{align*}
where we have used the well-known Pascal triangle relation. Thus we have shown
\begin{equation}\label{3.7}
S_m(n,r) + S_m(n,r+1) = S_m(n+1,r).
\end{equation}
We first observed the relation \eqref{3.7} by fixing small integers $m\geq 2$
and constructing tables for sufficient ranges of $n$ and $r$, using the
computer algebra package Maple. The above proof was then routine.

It is clear that the sequence $(S_m(0,r))_{r\geq 0}$ has only $-1$ and $1$ as
terms. First, when $m$ is odd, then by Lemma~\ref{lem:3.5}(2), no two terms
$-1$ can occur as neighbours. By \eqref{3.7} this means that that the sequence
$(S_m(1,r))_{r\geq 0}$ consists only of the terms 0 and 2. It now follows by
induction, with \eqref{3.7} as induction step, that for any $n\geq 1$ we have
$S_m(n,r)\geq0$ for all $r\geq 0$. This proves part (1) of the Proposition.

If $m$ is even, then by Lemma~\ref{lem:3.5}(3) there are two consecutive odd
binomial coefficients $\binom{k}{m}$, $\binom{k+1}{m}$. However, by \eqref{3.5}
not all $\binom{j}{m}$ can be odd; hence, keeping periodicity in mind,
there must be a triple of consecutive
binomial coefficients, the first of which is even, followed by two odd ones.
This, in turn, means that there is an integer $r\geq 1$ such that $S_m(0,r)=1$
and $S_m(0,r+1)=S_m(0,r+2)=-1$. The recurrence \eqref{3.7} then implies that
$S_m(1,r)=0$ and $S_m(1,r+1)=-2$, and applying \eqref{3.7} again, we have
$S_m(2,r)=-2$. This shows that the sequence $(f_{m,n}(-1))_{n\geq 1}$ is
not absolutely monotonic.
\end{proof}

We conclude this section with an easy consequence of the identity \eqref{4a.4}.
The second part of the following corollary will be used later, in Section~4.

\begin{corollary}\label{cor:3.6}
Let $m\geq 2$ and $\nu\geq 2$ be integers such that $2^{\nu-1}\leq m<2^{\nu}$.
Then the sequence $\big(f_{m,n}(-1)\big)_{n\geq 0}$ satisfies
\begin{equation}\label{3.8}
\Delta^{2^{\nu}}f_{m,n}(-1) = f_{m,n}(-1).
\end{equation}
If $m=2^k$ for some integer $k\geq 1$, then in addition to \eqref{3.8} we have
\begin{equation}\label{3.9}
\Delta^{2^k}f_{2^k,n}(-1) = -f_{2^k,n}(-1).
\end{equation}
\end{corollary}

\begin{proof}
We set $r=2^{\nu}$ and $z=-1$ in \eqref{4a.4}. Then with \eqref{3.2} we have
\[
\Delta^{2^{\nu}}f_{m,n}(-1)
=\sum_{j=0}^{n}\binom{n}{j}(-1)^{\binom{2^{\nu}+j}{m}}
=\sum_{j=0}^{n}\binom{n}{j}(-1)^{\binom{j}{m}} = f_{m,n}(-1),
\]
where we have used the fact that, by Lemma~\ref{3.5}(1), the binomial
coefficient $\binom{j}{m}$ is periodic modulo 2 with period $2^{\nu}$.

For \eqref{3.9}, we use again \eqref{4a.4} with $z=-1$, and this time
with $r=m=2^k$, obtaining
\begin{equation}\label{3.10}
\Delta^{2^k}f_{2^k,n}(-1) = \sum_{j=0}^{n}\binom{n}{j}(-1)^{\binom{2^k+j}{2^k}}.
\end{equation}
Now, by Lucas's congruence \eqref{3.5} we have
\begin{equation}\label{3.11}
\binom{j}{2^k} \equiv\begin{cases}
0\pmod{2},& 0\leq j\leq 2^k-1,\\
1\pmod{2},& 2^k\leq j\leq 2^{k+1}-1.
\end{cases}
\end{equation}
This, along with periodicity with period $2^{k+1}$ (see Lemma~\ref{lem:3.5}),
means that
\[
\binom{j+2^k}{2^k}\equiv\binom{j}{2^k}+1\pmod{2},
\]
which in turn shows that the right-hand side of \eqref{3.10} is
$-f_{2^k,n}(-1)$. This completes the proof.
\end{proof}

\section{Divisibility Properties}

In \cite{BDM3} it was shown that for any integer $k\geq 1$, the polynomial
$f_{2,2k+1}(z)$ is divisible by $z^k+1$. This gives rise to the question
whether there are similar divisibility results for polynomials $f_{m,n}(z)$
with other parameters $m$. Computations indicate that this is indeed the case
when $m$ is a power of 2, with certain additional restrictions. In fact, we
have the following result.

\begin{proposition}\label{prop:6.1}
Let $\mu\geq 1$ be a fixed integer, and suppose that the integer $k\geq 1$ is
not divisible by any odd prime $p<2^\mu$. Then
\[
z^k+1\quad\hbox{divides}\quad f_{2^{\mu},(k+1)2^{\mu}-1}(z).
\]
\end{proposition}

For the proof of this result we require the following two lemmas.

\begin{lemma}\label{lem:6.2}
For any integer $\mu\geq 1$, the exact power of $2$ in $2^\mu!$ is $2^\mu-1$.
\end{lemma}

\begin{proof}
Among various possible proofs, it is probably easiest to use the well-known
formula for the largest power of a prime in a factorial (see, e.g.,
\cite[p.~182]{NZM}), which in this case gives the exponent of 2 as
\[
\sum_{i\geq 1}\left\lfloor\frac{2^\mu}{2^i}\right\rfloor
= 2^{\mu-1}+2^{\mu-2}+\cdots+2+1 = 2^\mu-1,
\]
as claimed.
\end{proof}

\begin{lemma}\label{lem:6.3}
Let $\mu\geq 1$ be given. Then for any integer $j\geq 1$, the exact power of $2$
that divides
\begin{equation}\label{6.1}
\prod_{r=j}^{j+2^\mu-1}r\sum_{s=j}^{j+2^\mu-1}\frac{1}{s}
\quad\hbox{is}\quad 2^\mu-\mu-1,
\end{equation}
independent of $j$.
\end{lemma}

\begin{proof}
It is clear that among any $2^\mu$ consecutive integers, for instance those from
$j$ to $j+2^\mu-1$, we have that
\begin{align*}
2^{\mu-1}\quad\hbox{of them are}\quad &\equiv 1\pmod{2},\\
2^{\mu-2}\quad\hbox{of them are}\quad &\equiv 2\pmod{2^2},\\
&\vdots\\
\hbox{two of them are}\quad &\equiv 2^{\mu-2}\pmod{2^{\mu-1}},\\
\hbox{one of them is}\quad &\equiv 2^{\mu-1}\pmod{2^\mu},\quad\hbox{and}\\
\hbox{one of them is}\quad &\hbox{divisible by}\; 2^\mu.
\end{align*}
Comparing consecutive congruences, we see that the integers satisfying them
have to be distinct. Their total number is
$2^{\mu-1}+2^{\mu-2}+\cdots+1+1=2^{\mu}$, and thus they form a partition of
all the $2^{\mu}$ integers.

In \eqref{6.1}, consider the term where $s$ equals
the one integer in the given range that is divisible by $2^\mu$;
then the exact power of 2 in the product of all integers $r$,
$j\leq r\leq j+2^\mu-1$, without this $s$, is
\[
2^{\mu-2}\cdot 1 + 2^{\mu-3}\cdot 2 + \cdots + 2\cdot(\mu-2) + 1\cdot(\mu-1).
\]
This sums to $2^{\mu}-\mu-1$, which is easy to see, for instance by induction.
All the other $2^\mu-1$ products in the expression \eqref{6.1} are divisible by
higher powers of 2. This proves the statement of the lemma.
\end{proof}

\begin{proof}[Proof of Proposition~\ref{prop:6.1}]
We use the basic idea of the proof of Proposition~2.1 in \cite{BDM3}, which is
actually our case $\mu=1$. Using the definition \eqref{1.2}, we have
\begin{align}
f_{2^\mu,2^{\mu}k+2^\mu-1}(z)
&= \sum_{j=0}^{2^{\mu}k+2^\mu-1}\binom{2^{\mu}k+2^\mu-1}{j}z^{\binom{j}{2^\mu}}\label{6.2}\\
&= \sum_{j=0}^{2^{\mu-1}k+2^{\mu-1}-1}\binom{2^{\mu}k+2^\mu-1}{j}
\bigg(z^{\binom{j}{2^\mu}} + z^{\binom{2^{\mu}k+2^\mu-1-j}{2^\mu}}\bigg)\nonumber \\
&= \sum_{j=0}^{2^{\mu-1}k+2^{\mu-1}-1}\binom{2^{\mu}k+2^\mu-1}{j}z^{\binom{j}{2^\mu}}
\big(1 + z^{b_{\mu}(k,j)}\big),\nonumber
\end{align}
where
\begin{equation}\label{6.3}
b_{\mu}(k,j) := \binom{2^{\mu}k+2^{\mu}-1-j}{2^{\mu}} - \binom{j}{2^{\mu}}.
\end{equation}
We claim that if $k$ is not divisible by an odd prime $p<2^{\mu}$, then for all
integers $j$ with $0\leq j\leq 2^{{\mu}-1}k+2^{{\mu}-1}-1$, the integer $b_{\mu}(k,j)$ is
$k$ times an odd integer. But this would mean that
\[
1+z^k \mid 1+z^{b_{\mu}(k,j)},\qquad 0\leq j\leq 2^{\mu-1}k+2^{\mu-1}-1;
\]
this, with \eqref{6.2}, would prove the proposition.

It remains to prove our claim. We rewrite \eqref{6.3} as
\begin{align}
b_{\mu}(k,j) &= \frac{1}{2^{\mu}!}\left(\prod_{r=0}^{2^{\mu}-1}(2^{\mu}k-j+r)
-\prod_{r=0}^{2^{\mu}-1}(j-r)\right) \label{6.4}\\
&= \frac{1}{2^{\mu}!}
\left(-2^{\mu}k\prod_{r=0}^{2^{\mu}-1}(j-r)\sum_{s=0}^{2^{\mu}-1}\frac{1}{j-s}+\cdots\right),\nonumber
\end{align}
where the dots indicate multiples of $(2^{\mu}k)^{\nu}$, $\nu\geq 2$.
The second line of \eqref{6.4} is obtained from the first line by setting
$x:=2^{\mu}k$ and expanding the first product as a polynomial in $x$. Then the
constant coefficient is cancelled, the coefficient of $x$ is shown in the
second line, and the rest is represented by the dots.

Now by Lemma~\ref{lem:6.3}, the exact power of 2 that divides the expression in
parentheses on the right of \eqref{6.4}, excluding the factor $k$, is
$\mu+2^{\mu}-\mu-1$. Meanwhile, by Lemma~\ref{lem:6.2}, the exact power of 2 dividing
the denominator $2^{\mu}!$ is also $2^{\mu}-1$.

Finally we note that if $k$ is not divisible by any odd prime $p<2^{\mu}$, then
there cannot be any cancellation with the denominator $2^{\mu}!$. This means that
the integer $b_{\mu}(k,j)$ is divisible by $k$, and as we saw in the previous
paragraph, the quotient is an odd integer. This completes the proof.
\end{proof}

We can easily obtain the following consequence from Proposition~\ref{prop:6.1}.

\begin{corollary}\label{cor:6.4}
Let $\mu\geq 1$ be a fixed integer, and suppose that the integer $k\geq 1$ is
not divisible by any odd prime $p<2^\mu$. Then
\[
f_{2^{\mu},n}(z) \equiv 0\pmod{z^k+1}
\]
for infinitely many integers $n$.
\end{corollary}

\begin{proof}
Since $z^k+1$ divides $z^{k(2j+1)}+1$ for any integer $j\geq 0$, by
Proposition~\ref{prop:6.1} we see that
\[
f_{2^{\mu},n}(z) \equiv 0\pmod{z^k+1}\quad\hbox{for}\quad
n=\big(k(2j+1)+1\big)2^{\mu}-1.
\]
There are clearly infinitely many $j\geq 0$ such that $2j+1$ is not divisible
by an odd prime $p<2^{\mu}$; for instance, let $j$ run through all the multiples
of the product of all such primes. This proves the corollary.
\end{proof}

\noindent
{\bf Example~1.} Corollary~\ref{cor:6.4} shows that $z^k+1$ divides $f_{4,n}(z)$
for infinitely many $n$ when $k$ is not a multiple of 3. Similarly, $z^k+1$
divides $f_{8,n}(z)$ for infinitely many $n$ when $k$ is not divisible by 3, 5,
or 7.

\medskip
\noindent
{\bf Example~2.} On the other hand, for any $\mu\geq 1$ and any $j\geq 0$, we
have
\[
f_{2^{\mu},n}(z)\equiv 0\pmod{z^{2^j}+1}
\]
for infinitely many $n$. When $j=0$, we can actually show more:

\begin{corollary}\label{cor:6.5}
Given an integer $\mu\geq 1$, we have
\begin{equation}\label{6.5}
f_{2^{\mu},n}(z) \equiv 0\pmod{z+1}\quad\hbox{for}\quad
n=k\cdot 2^{\mu+1}-1,\quad k=0, 1, 2,\ldots
\end{equation}
\end{corollary}

\begin{proof}
By the definition \eqref{1.2} we have
\begin{equation}\label{6.6}
f_{2^{\mu},n}(-1)=\sum_{j=0}^n\binom{n}{j}(-1)^{b(j)},\qquad
b(j):=\binom{j}{2^{\mu}}.
\end{equation}
Now, by \eqref{3.11} we have, modulo 2,
\[
b(j)\equiv\begin{cases}
0, &\hbox{when}\;\; 0\leq j\leq 2^{\mu}-1,\\
1, &\hbox{when}\;\; 2^{\mu}\leq j\leq 2^{\mu+1}-1,
\end{cases}
\]
and by Lemma~\ref{lem:3.5}(1), this pattern continues with period $2^{\mu+1}$.
In particular, since $n=k\cdot 2^{\mu+1}-1$, this means that $b(j)$ and $b(n-j)$
have different parities, and thus
\[
(-1)^{b(j)} + (-1)^{b(n-j)} = 0. \quad j=0,1,\ldots,n.
\]
This, in turn, means that by \eqref{6.6} we have
\begin{align*}
f_{2^{\mu},n}(-1) &= \sum_{j=0}^{\frac{n-1}{2}}\binom{n}{j}(-1)^{b(j)}
+\sum_{j=0}^{\frac{n-1}{2}}\binom{n}{n-j}(-1)^{b(n-j)}\\
&= \sum_{j=0}^{\frac{n-1}{2}}\binom{n}{j}\left((-1)^{b(j)}+(-1)^{b(n-j)}\right)
=0,
\end{align*}
which completes the proof.
\end{proof}

We note that Corollary~\ref{cor:6.4} does not mean that we have no divisibility
in the exceptional cases. In fact, based on calculations we propose the
following

\begin{conjecture}\label{conj:6.6}
Let $\mu\geq 1$ be an integer. Then for any integer $k\geq 1$ there are
infinitely many $n$ such that $f_{2^{\mu},n}(z)\equiv 0\pmod{z^k+1}$.
\end{conjecture}

\section{Rational roots}

The existence of rational roots is obviously another divisibility property. In
the case of our polynomials $f_{m,n}(z)$ this question presents some interesting
challenges; we therefore devote a separate section to it. We begin with a
lemma which shows that we only need to consider one specific candidate.

\begin{lemma}\label{lem:4a.1}
Let $m\geq 2$ be an integer. The only possible rational root of $f_{m,n}(z)$
is $z_1=-1$, with the exception of the root $1-2^m$ of $f_{m,m}(z)$.
\end{lemma}

\begin{proof}
It is obvious from the first identity in \eqref{1.3} that $1-2^m$ is the only
root of $f_{m,m}(z)$. When $n<m$
then by the definition \eqref{1.2}, $f_{m,n}(z)$ is a positive integer. We
therefore assume that $n\geq m+1$.

In this case the polynomial $f_{m,n}(z)$ has leading coefficient 1, and
therefore any rational root is an integer dividing $f_{m,n}(0)$. Furthermore,
this divisor has to be negative since $f_{m,n}(z)$ has only nonnegative
coefficients. Suppose that $-g$ is such an integer solution, and for now we
assume that $g\geq 2$. Then with \eqref{1.2} we obtain
\begin{align*}
\left|f_{m,n}(-g)\right|
&\geq g^{\binom{n}{m}}-\sum_{j=0}^{n-1}\binom{n}{j}g^{\binom{j}{m}}
\geq g^{\binom{n}{m}}-g^{\binom{n-1}{m}}\big(2^n-1\big) \\
&> g^{\binom{n-1}{m}}\left(g^{\binom{n}{m}-\binom{n-1}{m}}-2^n\right)
= g^{\binom{n-1}{m}}\left(g^{\binom{n-1}{m-1}}-2^n\right).
\end{align*}
Since we assumed that $g\geq 2$, we then have
\begin{equation}\label{4a.1}
\left|f_{m,n}(-g)\right|
> 2^{\binom{n-1}{m}}\left(2^{\binom{n-1}{m-1}}-2^n\right).
\end{equation}
Now for $n\geq m+2$ and $m\geq 3$ we have
\[
\binom{n-1}{m-1} \geq \binom{n-1}{2} > n\qquad\hbox{for}\quad n\geq 5,
\]
where the second inequality is easy to verify, and the few cases with
$n\leq 4$ are easy to check by computation. Finally, when $n=m+1$, the second
identity in \eqref{1.3} shows that we only need to consider $z=-2$, and only
when $m$ is even, in which case we have $f_{m,m+1}(-2)=-3m-4$.

The case $m=2$ needs to be treated separately. In a similar way as in the
general case, but separating one more term from \eqref{1.2}, we have
\begin{align*}
\left|f_{2,n}(-g)\right|
&\geq g^{\binom{n}{2}}-n\cdot g^{\binom{n-1}{2}}
-g^{\binom{n-2}{2}}\big(2^n-n-1)\big) \\
&> g^{\binom{n-2}{2}}\left(g^{\binom{n}{2}-\binom{n-2}{2}}
-n\cdot g^{\binom{n-1}{2}-\binom{n-2}{2}}-2^n\right) \\
&= g^{\binom{n-2}{2}}\left(g^{2n-3}-n\cdot g^{n-2}-2^n\right)
\geq 2^{\binom{n-2}{2}}\left(2^{2n-3}-n\cdot 2^{n-2}-2^n\right)\\
&=2^{\binom{n-2}{2}}2^{n-2}\left(2^{n-1}-n-4\right) \geq 0
\end{align*}
for $n\geq 4$. Together with \eqref{4a.1} we have therefore shown that, when
$n\neq m$, the only possible rational root is $z_1=-1$, which concludes the
proof of the lemma.
\end{proof}

Lemma~\ref{lem:4a.1} shows that for a fixed $m\geq 2$ it suffices to consider
the sequence $(f_{m,n}(-1))_n$. By Proposition~\ref{prop:3.3} we know
that, when $m$ is odd, nothing more needs to be shown. However, since the
next result is of independent interest, we also include the case where $m$
is odd.

To motivate the following result, we consider the entries in Table~1.
Computations indicate that the sequence $(f_{2,n}(-1))$ satisfies the
recurrence relation $f_{2,n} = 2\,f_{2,n-1}-2\,f_{2,n-2}$,
where for simplicity we have deleted the argument $-1$, i.e., we put $f_{m,n}:=f_{m,n}(-1)$. Further, the
recurrences for $3\leq m\leq 6$ and $n$ sufficiently large, appear to be
\begin{align*}
f_{3,n} &= 4\,f_{3,n-1}-6\,f_{3,n-2}+4\,f_{3,n-3},\\
f_{4,n} &= 4\,f_{4,n-1}-6\,f_{4,n-2}+4\,f_{4,n-3}-2\,f_{4,n-4},\\
f_{5,n} &= 6\,f_{5,n-1}-14\,f_{5,n-2}+16\,f_{5,n-3}-10\,f_{5,n-4}+4\,f_{5,n-5}\\
f_{6,n} &= 8\,f_{6,n-1}-28\,f_{6,n-2}+56\,f_{6,n-3}-70\,f_{6,n-4}+56\,f_{6,n-5}\\
&\quad -28\,f_{6,n-6}+8\,f_{6,n-7}.
\end{align*}
If $p_m(x)$ denotes the corresponding characteristic polynomial, then we have,
along with their factorizations,
\begin{align*}
p_2(x)&=x^2-2x+2,\\
p_3(x)&=x^3-4x^2+6x-4 = (x^2-2x+2)(x-2),\\
p_4(x)&=x^4-4x^3+6x^2-4x+2,\\
p_5(x)&=x^5-6x^4+14x^3-16x^2+10x-4 =(x^4-4x^3+6x^2-4x+2)(x-2),\\
p_6(x)&=x^7-8x^6+28x^5-56x^4+70x^3-56x^2+28x-8\\
&=(x^4-4x^3+6x^2-4x+2)(x^2-2x+2)(x-2).
\end{align*}
To explain all this, we define the polynomials
\begin{align}
g_0(x) &:= x-2,\label{4a.2}\\
g_k(x) &:= (x-1)^{2^k}+1\qquad (k\geq 1).\label{4a.3}
\end{align}
By expanding the right-hand side of \eqref{4a.3} with the binomial theorem and
using, for instance, the congruence \eqref{3.5}, we see that $g_k(x)$ is a
2-Eisenstein polynomial for any $k\geq 0$, by which we mean that it satisfies
Eisenstein's criterion with the prime $p=2$; the polynomial is therefore
irreducible over the rationals. There is also a close connection with
cyclotomic polynomials; indeed, we can write
\begin{equation}\label{4a.3b}
g_k(x)=\Phi_{2^{k+1}}(x-1)\qquad (k\geq 1),
\end{equation}
and $g_0(x)=\Phi_1(x-1)$, where $\Phi_n(x)$ is the $n$th cyclotomic polynomial.
This provides another proof of the fact that all $g_k(x)$ are irreducible.

We are now ready to state the following result.

\begin{proposition}\label{prop:4a.2}
Let $p_m(x)$ be the characteristic polynomial of
$\big(f_{m,n}(-1)\big)_{n\geq 1}$, and let $m=2^{k_r}+\cdots+2^{k_1}$,
$k_r>\ldots>k_1\geq 0$, be the binary representation of $m\geq 2$. Then
\begin{enumerate}
\item[(1)] If $m=2^k$, then $p_m(x)=g_k(x)$.
\item[(2)] If $m$ is even and not a power of $2$, then
$p_m(x)=g_{k_r}(x)\cdots g_{k_1}(x)g_0(x)$.
\item[(3)] If $m$ is odd, then $p_m(x)=g_{k_r}(x)\cdots g_{k_1}(x)$.
\end{enumerate}
\end{proposition}

\noindent
{\bf Example.} For $m=2,3,\ldots,6$, we immediately obtain $p_2(x)=g_1(x)$ and
\[
p_3(x)=g_1(x)g_0(x),\; p_4(x)=g_2(x),\;
p_5(x)=g_2(x)g_0(x),\; p_6(x)=g_2(x)g_1(x)g_0(x),
\]
which is consistent with the polynomials listed above, before \eqref{4a.2}.

\begin{proof}[Proof of Proposition~\ref{prop:4a.2}]
(1) By \eqref{3.9} we have
\[
\sum_{\nu=0}^{2^k}(-1)^{\nu}\binom{2^k}{\nu}f_{2^k,n+2^k-\nu}(-1)
= -f_{2^k,n}(-1).
\]
This is therefore the recurrence relation for which
$g_k(x)$ is the characteristic polynomial, which proves part (1).

(2) We fix an even $m$, not a power of 2, and denote
\begin{equation}\label{4a.7}
p_m^{0}(x):=g_{k_r}(x)\cdots g_{k_1}(x),\quad\hbox{so that}\quad
p_m(x)=p_m^{0}(x)\cdot(x-2).
\end{equation}
Next we denote $A_m:=\{k_1,\ldots,k_r\}$, and for a subset
$A\subseteq A_m$ we define
\[
e(A) := \sum_{j\in A}2^j,
\]
so that in particular we have $e(\emptyset)=0$ and $e(A_m)=m$. Then by
\eqref{4a.7},
\begin{equation}\label{4a.8}
p_m^{0}(x)=\prod_{j\in A_m}\left((x-1)^{2^j}+1\right)
=\sum_{A\subseteq A_m}(x-1)^{e(A)},
\end{equation}
and with $g_0(x) = \big((x-1)-1\big)$,
\begin{equation}\label{4a.9}
p_m(x)= \sum_{A\subseteq A_m}\left((x-1)^{e(A)+1}-(x-1)^{e(A)}\right).
\end{equation}
Next we expand the terms in \eqref{4a.9} binomially and replace $x^j$ by
$f_{m,j+n}(-1)$. Then we use \eqref{4a.4} with $z=-1$ and $r=e(A)$, resp.\
$r=e(A)+1$, for all $A\subseteq A_m$, and the right-hand side of \eqref{4a.9}
becomes
\begin{align}
S_m(n) &:= \sum_{A\subseteq A_m}\left(
\sum_{j=0}^{n}\binom{n}{j}(-1)^{\binom{j+e(A)+1}{m}}
-\sum_{j=0}^{n}\binom{n}{j}(-1)^{\binom{j+e(A)}{m}}\right)\label{4a.10}\\
&= \sum_{j=0}^{n}\binom{n}{j}\sum_{A\subseteq A_m}
\left((-1)^{\binom{j+e(A)+1}{m}}-(-1)^{\binom{j+e(A)}{m}}\right).\nonumber
\end{align}
We are done if we can show that $S_m(n)=0$ for all $n\geq 1$, since then
$p_m(x)$ is indeed the characteristic polynomial for the sequence
$\big(f_{m,n}(-1)\big)_{n\geq 1}$.

To simplify the right-most term in \eqref{4a.10} we denote, for any integer
$r\geq 0$,
\[
\binom{r}{m}^*\equiv\binom{r}{m}\pmod{2},\qquad \binom{r}{m}^*\in\{0,1\}.
\]
Since obviously $(-1)^a = 1-2a$ for $a\in\{0,1\}$, we have
\[
(-1)^{\binom{r}{m}} = 1-2\binom{r}{m}^*\qquad (r=0,1,2,\ldots),
\]
and with \eqref{4a.10} we get
\begin{equation}\label{4a.11}
S_m(n) = 2\sum_{j=0}^{n}\binom{n}{j}\sum_{A\subseteq A_m}
\left(\binom{j+e(A)}{m}^*-\binom{j+e(A)+1}{m}^*\right).
\end{equation}
We recall that, by Lemma~\ref{3.5}(1), for a fixed $m$ with $2^{\nu-1}<m<2^\nu$,
the sequence $\binom{r}{m}^*$ is periodic with period $2^\nu$. Since
$m=2^{k_1}+\cdots+2^{k_r}$, by Lucas's congruence \eqref{3.5} we have
$\binom{j+e(A)}{m}^*=0$ unless all powers $2^{k_1},\ldots,2^{k_r}$ occur in the
binary expansion of $j+e(A)$. For each $j$ there is exactly one $A\subseteq A_m$
for which this is the case.
Indeed, let $B_j\subseteq A_m$ be the possibly empty subset containing all
$i\in A_m$ for which $2^i$ occurs in the binary expansion of $j$; then
$A=A_m\setminus B_j$. Similarly, for $j+e(A)+1$ we have the unique set
$A=A_m\setminus B_{j+1}$ for which the second binomial coefficient is 1.
These two values
``1" cancel, and thus the inner sum in \eqref{4a.11} vanishes for each $j$.
Hence $S_m(n)=0$ for all $n\geq 1$, which proves part (2).

(3) When $m$ is odd, the situation is similar to part (2), but with some
important differences. While in $p_m(x)$ we no longer consider the additional
factor $g_0(x)$, we now have $k_1=0$, and so $g_{k_1}(x)=x-2=(x-1)-1$.
Therefore we consider
\[
p_m(x) = g_{k_r}(x)\cdots g_{k_2}(x)\cdot(x-2),
\]
and with $A_m':=\{k_2,\ldots,k_r\}$ we have, as in \eqref{4a.9},
\[
p_m(x)= \sum_{A\subseteq A_m'}\left((x-1)^{e(A)+1}-(x-1)^{e(A)}\right),
\]
where, by definition, $e(A)$ is always even. Then \eqref{4a.11} holds as before,
with $A_m'$ in place of $A_m$.

To finish the proof, we use the same argument as in part (2) and note that
(since $e(A)$ is even) for each odd $j$ there
is exactly one $A\subseteq A_m'$ such that $\binom{j+e(A)}{m}^*=1$, while
$\binom{j+e(A)+1}{m}^*=0$ for all $A\subseteq A_m'$. Conversely, when $j$ is
even, there is exactly one $A\subseteq A_m'$ such that
$\binom{j+e(A)+1}{m}^*=1$, while $\binom{j+e(A)}{m}^*=0$ for all
$A\subseteq A_m'$.

This implies that the inner sum in \eqref{4a.11} is $(-1)^{j+1}$, and therefore,
by the binomial theorem, we have again $S_m(n)=0$ for all $n\geq 1$. This
completes part (3) of the proposition.
\end{proof}

We are now ready to prove the main result of this section.

\begin{proposition}\label{prop:4a.4}
Let $m\geq 2$ be an integer.
\begin{enumerate}
\item[(a)] $f_{m,m}(z)$ has the root $z_0=1-2^m$.
\item[(b)] When $m$ is odd and $n\geq 1$, then $f_{m,n}(z)$ has no other
rational roots.
\item[(c)] When $m$ is even but not a power of $2$, then $f_{m,n}(z)$ has no
other rational roots except, possibly, $z_1=-1$ for finitely many $n$.
\item[(d)] When $m=2^k$, $k\geq 1$, then $f_{m,2jm-1}(-1)=0$ for all
$j=1, 2,\ldots$, and there are at most finitely many other $n$ for which
$f_{m,n}(z)$ has a rational root.
\end{enumerate}
\end{proposition}

\begin{proof}
Statement (a) is obvious from the first identity in \eqref{1.3}. By
Lemma~\ref{lem:4a.1}, the only other possible rational root is $z_1=-1$. When
$m$ is odd, we use Proposition~\ref{prop:3.3}(1) which implies that the
sequence $(f_{m,n}(-1))_{n\geq 1}$ is increasing. But for $n\leq m-1$
these are positive constants, and also $f_{m,m}(-1)=2^m-2>0$; thus
$f_{m,n}(-1) >0$ for all $n\geq 1$, which proves part (b).

When $m$ is even and not a power of 2, we use Proposition~\ref{prop:4a.2}(2).
Since the polynomials $g_k(x)$, $k\geq 0$, are distinct and irreducible, the
characteristic polynomials $p_m(x)$ all have simple roots, one of which is
$x_0=2$. From \eqref{4a.3} we can explicitly determine all roots of $g_k(x)$
for $k\geq 1$, namely
\begin{equation}\label{4a.12}
1+\exp\left(\pm\frac{2j+1}{2^k}\pi i\right),\qquad j=0,1,\ldots,2^{k-1}-1,
\end{equation}
and from this it is not difficult to see that the respective moduli are
\[
2\cdot\cos\left(\frac{2j+1}{2^{k+1}}\pi\right) < 2.
\]
It follows from a well-known fact in the theory of linear recurrence relations
(see, e.g., \cite[p.~4]{EPSW} or \cite{Le}) that in this case, where $p_m(x)$
has only simple roots $x_0=2, x_1,\ldots,x_m$, we can write
\begin{equation}\label{4a.13}
f_{m,n}(-1) = a_02^n+a_1x_1^n+\cdots+a_mx_m^n.
\end{equation}
The coefficients $a_0,a_1,\ldots,a_m$ are constants that could be determined
by solving a suitable linear system, using $m+1$ terms of the sequence.
Since $x_0=2$ is the unique root of $p_m(x)$ with largest absolute value, it
can be shown by way of the method of Darboux (see, e.g., \cite[p.~310]{Ol}),
together with the theory of generating functions of linear recurrences (see,
e.g., \cite{Le}), that $f_{m,n}(-1)=O(2^n)$, and thus $a_0\neq 0$.
Alternatively, an explicit expression of $a_0$ can be found in \eqref{5.8} in
the next section. Now, since $|x_j|<2$ for all $j=1,\ldots,m$, we have
$f_{m,n}(-1)\neq 0$ if $n$ is sufficiently large. This proves part (c).

Finally, the first statement of part (d) is just a restatement of
Corollary~\ref{cor:6.5}, while the second statement follows from
Corollary~\ref{cor:5.5} in the next section.
\end{proof}

\begin{rem}\label{rem:4a.5}
{\rm (1) From \eqref{4a.12} it is also not difficult to see that the arguments
of the pair of roots belonging to $j$ are $\pm(2j+1)2^{-k-1}\pi$. So, in
particular, the two complex conjugate roots of $g_k(x)$ with largest modulus are
\begin{equation}\label{4a.14}
2\cdot\cos\left(\frac{\pi}{2^{k+1}}\right)\cdot
\exp\left(\pm\frac{\pi i}{2^{k+1}}\right),\quad\hbox{modulus}\quad
\beta_k:=2\cdot\cos\left(\frac{\pi}{2^{k+1}}\right).
\end{equation}
This means that the modulus of the largest roots gets very close
to 2 very quickly, as $k$ grows. For instance, the largest roots of $g_4(x)$
have modulus $2\cos(\pi/32)\simeq 1.99037.$

This fact, together with \eqref{4a.13}, explains why the sequence
$\big(f_{m,n}(-1)\big)_n$ displays a rather irregular behavior for some even
$m$. Here is a summary of our computations for even $m$, $1\leq m\leq 128$ and
$1\leq n\leq 5000$:
\begin{enumerate}
\item[(a)] $f_{12,n}(-1)<0$ for $24\leq n\leq 29$, and positive elsewhere.
\item[(b)] $f_{24,n}(-1)<0$ for $48\leq n\leq 62$ and $115\leq n\leq 123$, and
positive elsewhere.
\item[(c)] For $m=40$, 48, 56, 72, 80, 96, and 112, $f_{m,n}(-1)$ also has
intervals of negative values, not all beginning with $n=2m$.
\item[(d)] The values $f_{20,n}(-1)$ are all positive, but
$f_{20,44}(-1)<f_{20,42}(-1)$. Apart from (a)--(c) and $m=2^k$, this is the
only case for which monotonicity fails.
\item[(e)] For all other even $m$ that are not a power of 2, the sequence
$(f_{m,n}(-1))_{n\geq 1}$ is positive and strictly increasing.
\end{enumerate}

All these computations were done with Maple.
}
\end{rem}

\section{More on the sequence $f_{2^k,n}(-1)$}

We have seen in several places in Sections~3 and~4 that the case $m=2^k$ is
quite exceptional. We therefore devote this separate section to investigating
the sequence $f_{2^k,n}=f_{2^k,n}(-1)$ in greater detail, where $k\geq 1$ is
considered fixed. We recall that the sequence $(f_{2^{k},n})_{n\geq 0}$ is a linear recurrence sequence
with constant coefficients and with characteristic polynomial $g_k(x)$, as
defined in \eqref{4a.3}. We begin by obtaining the ordinary generating function
of this sequence.

\begin{proposition}\label{prop:5.1}
Let $k\geq 1$ be an integer. Then
\begin{equation}\label{5.1}
\sum_{n=0}^{\infty}f_{2^k,n}x^n
=\frac{1}{2x-1}\cdot\frac{x^{2^k}-(x-1)^{2^k}}{x^{2^k}+(x-1)^{2^k}},\qquad
|x|<\frac{1}{\beta_k},
\end{equation}
where $\beta_k=2\cos(\pi/2^{k+1})$.
\end{proposition}

\begin{proof}
Using the definition \eqref{1.2} and changing the order of summation, we
obtain
\[
S_k(x):=\sum_{n=0}^{\infty}f_{2^k,n}x^n
=\sum_{j=0}^{\infty}(-1)^{\binom{j}{2^k}}\sum_{n=j}^{\infty}\binom{n}{j}x^n.
\]
By absolute convergence for sufficiently small $x$ this is allowable. Upon
shifting the summation and using a well-known series evaluation (see, e.g.,
\cite[Eq.~(1.3)]{Go}), the inner sum becomes
\[
x^j\sum_{n=0}^{\infty}\binom{n+j}{j}x^n = \frac{x^j}{(1-x)^{j+1}},\qquad |x|<1,
\]
which gives
\begin{equation}\label{5.2}
S_k(x)=\sum_{j=0}^{\infty}(-1)^{\binom{j}{2^k}}\frac{x^j}{(1-x)^{j+1}}.
\end{equation}
The binomial coefficient in the exponent has already been evaluated in
\eqref{3.11}, and using periodicity with period $2^{k+1}$ (see again
Lemma~\ref{lem:3.5}), the series in \eqref{5.2} becomes
\begin{align}
S_k(x)&=\sum_{j=0}^{2^k-1}\sum_{\ell=0}^{\infty}
\left(\frac{x^{j+\ell 2^{k+1}}}{(1-x)^{j+1+\ell 2^{k+1}}}
-\frac{x^{j+2^k+\ell 2^{k+1}}}{(1-x)^{j+1+2^k+\ell 2^{k+1}}}\right)\label{5.3}\\
&=\sum_{j=0}^{2^k-1}\frac{x^j}{(1-x)^{j+1}}
\left(1-\frac{x^{2^k}}{(1-x)^{2^k}}\right)
\sum_{\ell=0}^{\infty}\frac{x^{\ell 2^{k+1}}}{(1-x)^{\ell 2^{k+1}}}.\nonumber
\end{align}
Now the finite sum in this last line evaluates as
\begin{equation}\label{5.4}
\frac{1}{1-x}\cdot\frac{1-\left(\frac{x}{1-x}\right)^{2^k}}{1-\frac{x}{1-x}}
= \frac{1}{1-2x}\cdot\frac{(1-x)^{2^k}-x^{2^k}}{(1-x)^{2^k}},
\end{equation}
while the infinite series in the same line has sum
\begin{equation}\label{5.5a}
\frac{1}{1-(x/(1-x))^{2^{k+1}}}
= \frac{(1-x)^{2^{k+1}}}{(1-x)^{2^{k+1}}-x^{2^{k+1}}},\qquad |x|<\frac{1}{2}.
\end{equation}
We substitute \eqref{5.4} and \eqref{5.5a} into \eqref{5.3}; then we get
\eqref{5.1} after some straightforward manipulations which include the
polynomial factorization
\[
(1-x)^{2^{k+1}}-x^{2^{k+1}}=\left((1-x)^{2^k}-x^{2^k}\right)
\left((1-x)^{2^k}+x^{2^k}\right)
\]

Finally we note that $x=1/2$ is a removable singularity of the right-hand side
of \eqref{5.1}. By analytic continuation, the identity \eqref{5.1} then holds
for all $x\in{\mathbb C}$ with $|x|<1/\beta_k$ since, by \eqref{4a.14},
$1/\beta_k$ is the smallest modulus of the roots of
\[
x^{2^k}+(x-1)^{2^k} = x^{2^k}g_k(\tfrac{1}{x}).
\]
This completes the proof of the proposition.
\end{proof}

As an application of \eqref{5.1} we set $x=\frac{1}{2}$, which still lies
inside the circle of convergence. Then after some easy manipulations (e.g.,
using L'Hospital's Rule on the right-hand side of \eqref{5.1}),
we get the following somewhat surprising series evaluations.

\begin{corollary}\label{cor:5.2}
For any integer $k\geq 1$, we have
\[
\sum_{n=0}^{\infty}f_{2^k,n}\cdot\big(\tfrac{1}{2}\big)^n = 2^k.
\]
\end{corollary}

The next result gives an explicit formula for all $f_{2^k,n}$; it can also be
seen as a refinement of Corollary~\ref{cor:6.5}.

\begin{proposition}\label{prop:5.3}
For any integers $k\geq 1$ and $n\geq 0$ we have
\begin{equation}\label{5.5}
f_{2^k,n} = 2^{1-k}\sum_{j=1}^{2^{k-1}}
\frac{\left(2\cos(\frac{2j-1}{2^{k+1}}\pi)\right)^n}{\sin(\frac{2j-1}{2^{k+1}}\pi)}
\cdot\sin\left((n+1)\frac{2j-1}{2^{k+1}}\pi\right).
\end{equation}
\end{proposition}

Before proving this result, we give the two smallest cases as illustrations.
For this, we have used some well-known special values for sine and cosine.

\begin{corollary}\label{cor:5.4}
For all integers $n\geq 0$ we have
\begin{align*}
f_{2,n} &= \big(\sqrt{2}\big)^{n+1}\sin\left(\tfrac{n+1}{4}\pi\right),\\
f_{4,n} &= \frac{1}{\sqrt{2}}\big(2+\sqrt{2}\big)^{\frac{n+1}{2}}\sin\big(\tfrac{n+1}{8}\pi\big)
+\frac{1}{\sqrt{2}}\big(2-\sqrt{2}\big)^{\frac{n+1}{2}}\sin\big(\tfrac{3(n+1)}{8}\pi\big).
\end{align*}
\end{corollary}

\begin{proof}[Proof of Proposition~\ref{prop:5.3}]
By the theory of linear recurrence relations (see, e.g., \cite[p.~4]{EPSW} or
\cite{Le}),
and since the characteristic polynomial $g_k(x)$ has only simple roots, we have
\begin{equation}\label{5.6}
f_{2^k,n} = \sum_{j=1}^{2^k}a_j^{(k)}\cdot\big(x_j^{(k)}\big)^n,
\end{equation}
where $a_j^{(k)}$, $j=1,2,\ldots,2^k$, are constant coefficients, and
$x_j^{(k)}$, $j=1,2,\ldots,2^k$, are the roots of $g_k(x)$. As we saw in
\eqref{4a.13} and in Remark~\ref{rem:4a.5}(1), we have
\begin{equation}\label{5.7}
x_j^{(k)} = 1+\exp\left(\frac{2j-1}{2^{k}}\pi i\right)
=2\cos\left(\frac{2j-1}{2^{k+1}}\pi\right)
\exp\left(\frac{2j-1}{2^{k+1}}\pi i\right).
\end{equation}
To determine the coefficients $a_j^{(k)}$, we use \eqref{5.6} together with
\eqref{5.4}, to set up a linear system of $2^k$ equations for
$n=0,1,\ldots,2^k-1$ (the matrix of this system is a Vandermonde matrix). We
did this for some small $k$ and found, conjecturally, that
\begin{equation}\label{5.8}
a_j^{(k)} = \frac{2^{1-k}}{1-\exp\left(-\frac{2j-1}{2^{k}}\pi i\right)}
= \frac{-i\cdot x_j^{(k)}}{2^k\sin\left(\frac{2j-1}{2^{k}}\pi\right)}.
\end{equation}
Pairing the product of \eqref{5.7} and \eqref{5.8} for each $j$ with that of
$2^k+1-j$, $j=1,2,\ldots,2^{k-1}$, we obtain \eqref{5.5} from \eqref{5.6}.
In order to prove this in general, it remains to show that for each $k\geq 1$,
the right-hand side of \eqref{5.5} equals $2^n$ for all $n=0,1,\ldots,2^k-1$,
or equivalently
\begin{equation}\label{5.9}
\sum_{j=1}^{2^{k-1}}\frac{\sin((n+1)\alpha_j)}{\sin(\alpha_j)}\cos^n(\alpha_j)
=2^{k-1},\qquad \alpha_j:=\frac{2j-1}{2^{k+1}}\pi.
\end{equation}
This identity actually holds in greater generality. We are going to use
Chebyshev polynomials of the second kind, $U_n(x)$, defined by
\[
U_n(\cos\theta) = \frac{\sin((n+1)\theta)}{\sin\theta}
\]
(see, e.g., \cite[Eq.~(1.23)]{Riv}), and we will show that
\begin{equation}\label{5.10}
\sum_{j=1}^{m}U_n\left(\cos\big(\tfrac{2j-1}{4m}\pi\big)\right)
\cos^n\big(\tfrac{2j-1}{4m}\pi\big) = m,\quad 0\leq n\leq 2m-1.
\end{equation}
Then \eqref{5.9} immediately follows from \eqref{5.10}, with $m=2^{k-1}$.

When $n=0$, then \eqref{5.10} is trivially true. To prove \eqref{5.10} for
$n\geq 1$, we use the well-known explicit formula
\[
U_n(x)
= \sum_{\nu=0}^{\lfloor n/2\rfloor}(-1)^{\nu}\binom{n-\nu}{\nu}(2x)^{n-2\nu}
\qquad (n\geq 1);
\]
see, e.g., \cite[p.~39]{Riv}. Substituting this into \eqref{5.10} and changing
the order of summation, we see that \eqref{5.10} holds if we can show that
\begin{equation}\label{5.11}
\sum_{\nu=0}^{\lfloor n/2\rfloor}(-1)^{\nu}\binom{n-\nu}{\nu}2^{n-2\nu}
\sum_{j=1}^{m}\cos^{2n-2\nu}\big(\tfrac{2j-1}{4m}\pi\big) = m.
\end{equation}
The inner sum in \eqref{5.11} is easy to reduce to a known sum; indeed, if we
rewrite it as
\[
\sum_{j=1}^{2m}\cos^{2n-2\nu}\big(\tfrac{j}{4m}\pi\big)
-\sum_{j=1}^{m}\cos^{2n-2\nu}\big(\tfrac{2j}{4m}\pi\big),
\]
we can use the identity 4.4.2.11 in \cite[p.~640]{PBM} twice, obtaining
\[
\sum_{j=1}^{m}\cos^{2n-2\nu}\big(\tfrac{2j-1}{4m}\pi\big)
=\frac{4m}{2^{2n-2\nu+1}}\binom{2n-2\nu}{n-\nu}
-\frac{2m}{2^{2n-2\nu+1}}\binom{2n-2\nu}{n-\nu}
\]
(valid for $n-\nu<2m$), so that the left-hand side of \eqref{5.11} becomes
\[
\frac{m}{2^n}\sum_{\nu=0}^{\lfloor n/2\rfloor}(-1)^{\nu}\binom{n-\nu}{\nu}\binom{2n-2\nu}{n-\nu}
=\frac{m}{2^n}\sum_{\nu=0}^{\lfloor n/2\rfloor}(-1)^{\nu}\binom{n}{\nu}\binom{2n-2\nu}{n},
\]
where it is easy to check that the two products of binomial coefficients are
identical. Finally, the sum on the right has the known evaluation $2^n$; see,
e.g., \cite[Eq.~(3.117)]{Go}. Thus we have shown that \eqref{5.11} holds, which
completes the proof.
\end{proof}

The following result is our main application of Proposition~\ref{prop:5.3}; in
fact, it was already used in the proof of Proposition~\ref{prop:4a.4}(d).

\begin{corollary}\label{cor:5.5}
Let $k\geq 1$ and $1\leq r\leq 2^{k+1}-1$ be fixed integers. Then for all
$\nu\geq 2^{3k-1}/\pi^2$ we have
\[
(-1)^{\nu}f_{2^k,n}>0,\quad\hbox{where}\quad n=\nu\cdot 2^{k+1}+r-1.
\]
\end{corollary}

\begin{proof}
For $k=1$ and $n=4\nu+r-1$ we have by Corollary~\ref{cor:5.4},
\[
f_{2,n} = \big(\sqrt{2}\big)^{n+1}\sin\left(\tfrac{4\nu+r}{4}\pi\right)
=(-1)^{\nu}\big(\sqrt{2}\big)^{4\nu+r}\sin\left(\tfrac{r\pi}{4}\right),
\]
and since $1\leq r\leq 3$, the sine term on the right is positive. Hence the
statement is true for $k=1$ and all $\nu\geq 0$.

Now let $k\geq 2$. We are going to use \eqref{5.5}, and first note that
\begin{align*}
\sin\left((n+1)\frac{2j-1}{2^{k+1}}\pi\right)
&=\sin\left(\nu(2j-1)\pi+\frac{r(2j-1)}{2^{k+1}}\pi\right)\\
&=(-1)^{\nu}\sin\left(\frac{r(2j-1)}{2^{k+1}}\pi\right),
\end{align*}
so that
\begin{equation}\label{5.12}
f_{2^k,n} = 2^{n+1-k}(-1)^{\nu}\sum_{j=1}^{2^{k-1}}
\cos^n\left(\frac{2j-1}{2^{k+1}}\pi\right)
\cdot\frac{\sin(\frac{r(2j-1)}{2^{k+1}}\pi)}{\sin(\frac{2j-1}{2^{k+1}}\pi)},
\end{equation}
with $n=\nu\cdot 2^{k+1}+r-1$. Now let $S_n$ be the sum on the right of
\eqref{5.12}. We now use the fact that $|\sin(r\alpha)/\sin(\alpha)|\leq r$ for
any $\alpha\in{\mathbb R}$ and integer $r\geq1$. This can be seen, for
instance, by combining the identities (1.23) and (1.24) in \cite[pp.~7-8]{Riv}.
Then we have the estimate
\begin{align}
S_n &\geq \cos^n\left(\frac{\pi}{2^{k+1}}\right)
-\sum_{j=2}^{2^{k-1}}\cos^n\left(\frac{3\pi}{2^{k+1}}\right)\cdot r\label{5.13}\\
&\geq  \cos^n\left(\frac{\pi}{2^{k+1}}\right)
-\cos^n\left(\frac{3\pi}{2^{k+1}}\right)\big(2^{k-1}-1\big)\big(2^{k+1}-1\big)\nonumber\\
&>\cos^n\left(\frac{3\pi}{2^{k+1}}\right)
\left(\left(\frac{\cos(\pi/2^{k+1})}{\cos(3\pi/2^{k+1})}\right)^n-2^{2k}\right).\nonumber
\end{align}
We now estimate the quotient of cosines in this last expression. For ease of
notation we set $\alpha:=\pi/2^{k+1}$, and first note that
$\alpha\leq\frac{\pi}{8}<\frac{4}{5}$ for $k\geq 2$. By the Maclaurin expansion
for cosine we have
\[
\cos\alpha>1-\tfrac{1}{2}\alpha^2 \quad\hbox{and}\quad
\cos(3\alpha)<1-\tfrac{1}{2}(3\alpha)^2+\tfrac{1}{24}(3\alpha)^4.
\]
So we get
\[
\frac{\cos(\alpha)}{\cos(3\alpha)}
> \frac{1-\frac{1}{2}\alpha^2}{1-\tfrac{1}{2}(3\alpha)^2+\tfrac{1}{24}(3\alpha)^4}
> 1+4\alpha^2,
\]
where it is straightforward to verify that the right inequality holds for
$0<\alpha<\frac{4}{5}$. Thus, using $n=\nu\cdot 2^{k+1}+r-1$,
\[
\left(\frac{\cos(\alpha)}{\cos(3\alpha)}\right)^n
> \left(1+4\cdot\frac{\pi^2}{2^{2k+2}}\right)^n
> 1 + \nu\cdot2^{k+1}\cdot 4\cdot\frac{\pi^2}{2^{2k+2}}
> \frac{\nu\pi^2}{2^{k-1}}.
\]
Hence, by \eqref{5.13} we have $S_n>0$ when $\nu\pi^2\geq 2^{3k-1}$, and with
\eqref{5.12} this completes the proof.
\end{proof}

It is clear from this proof that the lower bound for $\nu$ could be somewhat
improved, but also, we conjecture that the statement of Corollary~5.5 holds for
all $\nu\geq 0$. By numerical computation we checked that our conjecture is
true for $k\leq 5$. In fact, at the end of this section we propose a stronger
conjecture.

As another consequence of Proposition~\ref{prop:5.3} we obtain a proof of the
observation that in each sequence $(f_{2^k,n})_{n\geq 0}$, any two terms that
immediately precede a zero term are identical; see also Table~1. A second,
related, identity can be obtained in a similar way. We recall
that $f_{2^k,\nu\cdot 2^{k+1}-1}=0$ for all integers $k,\nu\geq 1$, a fact that
is also obvious from \eqref{5.5}.

\begin{corollary}\label{cor:5.6}
For all integers $k,\nu\geq 1$ we have
\begin{align*}
f_{2^k,\nu\cdot 2^{k+1}-2} &= f_{2^k,\nu\cdot 2^{k+1}-3},\\
f_{2^k,(2\nu-1)2^{k+1}-1} &= 2f_{2^k,(2\nu-1)2^{k+1}-2}.
\end{align*}
\end{corollary}

\begin{proof}
To obtain the first identity we show that, in fact, for a fixed $k\geq 1$ the
corresponding summands on the right of \eqref{5.5} have the same values for
each $j=1,2,\ldots,2^{k-1}$. This is equivalent to
\[
2\cos(\alpha_j)\sin\big((\nu\cdot 2^{k+1}-1)\alpha_j\big)
=\sin\big((\nu\cdot 2^{k+1}-2)\alpha_j\big),\quad
\alpha_j:=\frac{2j-1}{2^{k+1}}\pi.
\]
But this identity is easy to verify by way of some elementary trigonometric
identities. The second identity can be obtained in an analogous way.
\end{proof}

It follows from the definition \eqref{1.2} that $f_{2^k,n}=2^n$ for
$0\leq n\leq 2^k-1$. We can extend this as follows. This is also related to
Corollary~\ref{cor:5.5} with $\nu=1$.

\begin{proposition}\label{prop:5.7}
Let $k\geq 2$ be an integer. Then the sequence $\big(f_{2^k,n}\big)_{n\geq 0}$
is positive and nondecreasing for $0\leq n\leq 2^{k+1}-2$.
\end{proposition}

\begin{proof} For $0\leq n\leq 2^k-1$, the statement is clear by the remark
just before the proposition. For $2^k\leq n\leq 2^{k+1}-2$,
we use \eqref{3.11}, obtaining
\begin{equation}\label{5.12a}
f_{2^k,n}=\sum_{j=0}^n\binom{n}{j}(-1)^{\binom{j}{2^k}}
=\sum_{j=0}^{2^k-1}\binom{n}{j}-\sum_{j=2^k}^n\binom{n}{j}.
\end{equation}
For $n$ in the given range, we see that each negative binomial coefficient on
the right is canceled by its positive counterpart, with at least one positive
term remaining. This proves the positivity claim.

Next, by the left equality of \eqref{5.12a}, or by \eqref{4a.4} with $r=1$ and
$z=-1$, we have
\begin{equation}\label{5.13a}
f_{2^k,n+1}-f_{2^k,n}=\sum_{j=0}^n\binom{n}{j}(-1)^{\binom{j+1}{2^k}}
=\sum_{j=0}^{2^k-2}\binom{n}{j}-\sum_{j=2^k-1}^n\binom{n}{j},
\end{equation}
where we have used \eqref{3.11} again. We now argue just as in the first part
of this proof: When $n$ is such that $2^k-1\leq n\leq 2^{k+1}-4$, then each
negative binomial coefficient is canceled by its positive counterpart, with at
least one positive term remaining. Hence $f_{2^k,n}$ is strictly increasing
for $n\leq 2^{k+1}-3$. Finally, the right-hand side of \eqref{5.13a} vanishes
for $n=2^{k+1}-3$; this also follows from Corollary~\ref{cor:5.5}.
\end{proof}

Computations indicate that the behaviour of the sequence
$(f_{2^k,n})_{n\geq 0}$ proved in Proposition~\ref{prop:5.7} holds for each
interval between the zeros that occur at all $n=(\nu+1)\cdot 2^{k+1}-1$,
$\nu=0, 1,2,\ldots$.

\begin{conjecture}\label{conj:5.8}
Let $k\geq 1$ and $\nu\geq 0$ be integers. If $\nu\cdot 2^{k+1}\leq n\leq(\nu+1)\cdot 2^{k+1}-2$,
then $(-1)^{\nu}f_{2^k,n}>0$, and the sequence
$\big((-1)^{\nu}f_{2^k,n}\big)_n$ is strictly increasing in this interval,
with the exception of the final two terms which are equal.
\end{conjecture}

Further supporting evidence for this conjecture is given by
Corollaries~\ref{cor:5.5} and~\ref{cor:5.6}, where the former shows that the
alternating sign structure is true, at least for sufficiently large $\nu$,
depending on $k$.

\section{Some irreducibility results}

Computations with Maple suggest that, apart from the factors $z^k+1$ exhibited
in the previous section, and the rational roots in Proposition~3.7(a), all other
polynomials $f_{m,n}(z)$ are irreducible. While we are unable to prove this in
general, we have the following result. For the remainder of this paper,
``irreducible" will mean irreducible over $\Q$.

\begin{proposition}\label{prop:4.1}
Let $p$ be an odd prime, $d$ an integer with $1\leq d\leq p-1$, and suppose
that
\begin{equation}\label{4.1}
\sum_{k=1}^d\frac{(-1)^{k-1}}{k} \not\equiv 0\pmod{p}.
\end{equation}
Then for every $n=j(p-1)p$, where $j=1,2,\ldots$ and $p\nmid j$, the polynomial
$f_{n-d,n}(z)$ is $p$-Eisenstein and thus irreducible.
If, furthermore, $p\equiv\pm 1\pmod{8}$, then the conclusion holds for all
$n=j(p-1)p/2$, with $j$ as above.
\end{proposition}

\begin{proof}
By \eqref{1.2} we have
\begin{equation}\label{4.2}
f_{n-d,n}(z)=z^{\binom{n}{n-d}}+\binom{n}{1}z^{\binom{n-1}{n-d}}+\cdots
+\binom{n}{d}z+\sum_{k=0}^{n-d-1}\binom{n}{k}.
\end{equation}
We now consider
\[
\binom{n}{r} = \frac{n(n-1)\cdots(n-r+1)}{r!},\qquad 1\leq r\leq d\leq p-1.
\]
If $n$ is a multiple of $p$, we see that there is no cancellation, and
thus $p\mid\binom{n}{r}$. Therefore, to prove that $f_{n-d,n}(z)$ is
$p$-Eisenstein, it remains to show that
\begin{equation}\label{4.3}
p\|\sum_{k=0}^{n-d-1}\binom{n}{k},
\end{equation}
that is, $p$ but not $p^2$ divides the sum on the right. To do so, we note that
\begin{equation}\label{4.4}
\sum_{k=0}^{n-d-1}\binom{n}{k} = 2^n-1-\binom{n}{1}-\cdots-\binom{n}{d}.
\end{equation}
First, by Fermat's little theorem, we have for $n=j(p-1)p$,
\begin{equation}\label{4.5}
2^n=\big(2^{p-1}\big)^{jp} = (1+\nu p)^{jp} = 1+jp\nu p+O(p^2)
\equiv 1\pmod{p^2}.
\end{equation}
If $p\equiv\pm 1\pmod{8}$, then 2 is a quadratic residue modulo $p$, and by
Euler's criterion we have $2^{(p-1)/2}\equiv 1\pmod{p}$. Then, just as in
\eqref{4.5}, we get
\begin{equation}\label{4.6}
2^n \equiv 1\pmod{p^2}\qquad\hbox{for}\quad n=j(p-1)p/2.
\end{equation}
Next, when $n=sp$, $p\nmid s$, then for $1\leq k\leq d$ we have
\begin{align*}
\binom{n}{k} &= \frac{sp}{k!}(sp-1)(sp-2)\cdots(sp-k+1)\\
&\equiv \frac{sp}{k!}(-1)^{k-1}(k-1)! = sp\frac{(-1)^{k-1}}{k}\pmod{p^2}.
\end{align*}
This, together with \eqref{4.4} and with \eqref{4.5}, resp.\ \eqref{4.6},
shows that
\[
\sum_{k=0}^{n-d-1}\binom{n}{k}
\equiv -sp\sum_{k=1}^{d}\frac{(-1)^{k-1}}{k}\pmod{p^2}.
\]
Hence \eqref{4.1} implies \eqref{4.3}, and the proof is complete.
\end{proof}

\noindent
{\bf Example.} Let $d=3$ and $p=5$. Then $j=1$ gives $n=20$, and
\[
f_{17,20}(z) = z^{1140}+20\,z^{171}+190\,z^{18}+1140\,z+1\,047\,225.
\]
As we can see, $5^2$ divides the constant coefficients, so this polynomial is
not 5-Eisenstein. And indeed, we have
$1-\frac{1}{2}+\frac{1}{3}=\frac{5}{6}$, so \eqref{4.1} does not hold.

On the other hand, $p=7$ does satisfy this condition, and since
$7\equiv -1\pmod{8}$, Proposition~\ref{prop:4.1} applies to $n=21$. In fact, it
is easily seen that
\[
f_{18,21}(z) = z^{1330}+21\,z^{190}+210\,z^{19}+1330\,z+2\,095\,590
\]
is indeed 7-Eisenstein. Finally, we
note that, although $f_{17,20}(z)$ does not satisfy the Eisenstein criterion,
one can verify by computer algebra (in our case, using Maple) that it is
irreducible.

\medskip
In the cases $d=1$ and $d=2$, the condition \eqref{4.1} becomes irrelevant,
and we can state the following corollary,

\begin{corollary}\label{cor:4.2}
Let $p$ be an odd prime, and let $n=j(p-1)p$, resp.\ $n=j(p-1)p/2$ when
$p\equiv\pm 1\pmod{8}$, where $j=1,2,\ldots$ and $p\nmid j$. Then
$f_{n-1,n}(z)$ and $f_{n-2,n}(z)$ are irreducible.
\end{corollary}

The next corollary has an unexpected connection with Wieferich primes, which
are closely related to Fermat quotients. For an odd prime $p$ and an integer
$a\geq 2$ with $p\nmid a$, the {\it Fermat quotient to base} $a$ is defined by
\[
q_p(a) := \frac{a^{p-1}-1}{p}.
\]
Fermat's little theorem implies that this is an integer. A prime $p$ that
satisfies $q_p(2)\equiv 0\pmod{p}$ is called a {\it Wieferich prime}. These
primes played an
important role in the classical theory of Fermat's last theorem; see, e.g.,
\cite{Ri}. Only two such primes are known, namely $p=1093$ and $p=3511$. The
latest published search \cite{DK} for Wieferich primes went up to
$6.7\times 10^{15}$, while the current record stands at $6\times 10^{17}$;
see \cite{Fis}. It is not known whether there are infinitely many Wieferich
primes, or even whether there are infinitely many non-Wieferich primes;
see \cite{GM}.

\begin{corollary}\label{cor:4.3}
Let $p$ be an odd non-Wieferich prime, and let $d=p-1$, $d=(p-1)/2$, or
$d=\lfloor p/3\rfloor$. Then $f_{n-d,n}(z)$ is irreducible for all $n=j(p-1)p$, resp.\ $n=j(p-1)p/2$ when $p\equiv\pm 1\pmod{8}$, where $j=1,2,\ldots$ and
$p\nmid j$.
\end{corollary}

\begin{proof}
To apply Proposition~\ref{prop:4.1}, it remains to verify \eqref{4.1}. First we
note that
\begin{equation}\label{4.7}
\sum_{k=1}^d\frac{(-1)^{k-1}}{k}
= \sum_{k=1}^d\frac{1}{k} - \sum_{k=1}^{\lfloor d/2\rfloor}\frac{1}{k}.
\end{equation}
We now recall the classical congruences
\begin{align}
\sum_{k=1}^{p-1}\frac{1}{k}&\equiv 0\pmod{p},\qquad
\sum_{k=1}^{(p-1)/2}\frac{1}{k}\equiv -2q_p(2)\pmod{p},\label{4.8}\\
\sum_{k=1}^{\lfloor p/3\rfloor}\frac{1}{k}&\equiv -\frac{3}{2}q_p(3)\pmod{p},\qquad
\sum_{k=1}^{\lfloor p/4\rfloor}\frac{1}{k}\equiv -3q_p(2)\pmod{p},\label{4.9}\\
\sum_{k=1}^{\lfloor p/6\rfloor}\frac{1}{k}
&\equiv -2q_p(2)-\frac{3}{2}q_p(3)\pmod{p}.\label{4.10}
\end{align}
All these congruences have well-known extensions modulo $p^2$ and $p^3$.
The left congruence in \eqref{4.8} follows from the fact that
$\{1, 1/2,\ldots,1/(p-1)\}$ forms a reduced residue system modulo $p$, the sum
of which is divisible by $p$. The right-hand congruence in \eqref{4.8} goes back
to Eisenstein in 1850. All are special cases of congruences in \cite{Le2}; see
also \cite[p.~155]{Ri}. Combining them with \eqref{4.7}, we see that
\begin{align*}
\sum_{k=1}^{p-1}\frac{(-1)^{k-1}}{k}&\equiv 2q_p(2)\pmod{p},\qquad
\sum_{k=1}^{(p-1)/2}\frac{(-1)^{k-1}}{k}\equiv q_p(2)\pmod{p},\\
\sum_{k=1}^{\lfloor p/3\rfloor}\frac{(-1)^{k-1}}{k}&\equiv 2q_p(2)\pmod{p}.
\end{align*}
These cannot vanish modulo $p$ unless $p$ is a Wieferich prime.
\end{proof}

\section*{Acknowledgments}

We would like to thank the anonymous referee for numerous suggestions which
helped improve this paper.

\end{document}